\documentclass[reqno, 11pt]{amsart}
\usepackage{amsmath,amsfonts,amssymb,amsthm,epsfig,epstopdf,url,array,comment,hyperref,mathrsfs,mathtools}
\usepackage{hyperref}
\usepackage{setspace}
\usepackage{graphicx}
\usepackage{float}
\usepackage{tikz-cd}
\usepackage[cmtip,all]{xy}
\newcommand{\longsquiggly}{\xymatrix{{}\ar@{<~>}[r]&{}}}
\usepackage{tabularray}
\mathtoolsset{showonlyrefs}

\theoremstyle{plain}
\newtheorem{thm}{Theorem}[section]

\newtheorem{prop}[thm]{Proposition}
\newtheorem{cor}[thm]{Corollary}

\theoremstyle{definition}
\newtheorem{definition}[thm]{Definition}
\newtheorem{con}[thm]{Conjecture}

\newtheorem{qstn}[thm]{Question}
\newtheorem*{notation}{Notation}

\theoremstyle{remark}
\newtheorem{rem}{Remark}
\newtheorem*{rems}{Remarks}

\newcommand{\R}{\mathbb{R}}
\renewcommand{\d}{\textnormal{ d}}

\newcommand{\grad}{\nabla}

\newcommand{\tr}{\textnormal{tr}}
\newcommand{\hidethis}[1]{}

\newcommand{\E}{\mathbb{E}}

\newcommand{\prob}{\mathbb{P}}

\newcommand{\var}{\textnormal{Var}}

\newcommand{\del}{\partial}

\renewcommand{\div}{\textnormal{div}}
\newcommand{\divtil}{\widetilde{\textnormal{div}}}
\newcommand{\gtil}{\widetilde{\mathcal{G}}}

\newcommand{\RR}{\mathbb{R}}

\begin{document}

\title[Entropic and functional forms]{Entropic and functional forms of the dimensional Brunn--Minkowski inequality in Gauss space}
\author{Gautam Aishwarya}
\address{Department of Mathematics, Michigan State University, East Lansing 48824, USA.}
\email{aishwary@msu.edu}
\author{Dongbin Li}
\address{Faculty of Science - Mathematics and Statistical Sciences, University of Alberta, Edmonton, AB T6G 2R3, Canada.}
\email{dongbin@ualberta.ca}
\begin{abstract}
Given even strongly log-concave random vectors $X_{0}$ and $X_{1}$ in $\R^n$, we show that a natural joint distribution $(X_{0},X_{1})$ satisfies
\begin{equation}
    e^{ - \frac{1}{n}D ((1-t)X_{0} + t X_{1} \Vert Z)} \geq (1-t) e^{ - \frac{1}{n}D (X_{0} \Vert Z)} + t e^{ - \frac{1}{n}D ( X_{1} \Vert Z)},  
\end{equation}
where $Z$ is distributed according to the standard Gaussian measure $\gamma$ on $\R^n$, $t \in [0,1]$, and $D(\cdot \Vert Z)$ is the Gaussian relative entropy. This extends and provides a different viewpoint on the corresponding geometric inequality proved by Eskenazis and Moschidis \cite{EskenazisMoschidis21}, namely that 
\begin{equation} 
\gamma \left( (1-t) K_{0} + t K_{1} \right)^{\frac{1}{n}} \geq (1-t) \gamma (K_{0})^{\frac{1}{n}} + t \gamma (K_{1})^{\frac{1}{n}},  
\end{equation}
when $K_{0}, K_{1} \subseteq \R^n$ are origin-symmetric convex bodies. As an application, using Donsker--Varadhan duality, we obtain Gaussian Borell--Brascamp--Lieb inequalities applicable to even log-concave functions, which serve as functional forms of the Eskenazis--Moschidis inequality.  
\end{abstract}
\thanks{{\it MSC classification}:
37C10, %Dynamics induced by flows and semiflows%
            94A17, %Measures of information, entropy%
            52A40, %Inequalities and extremum problems involving convexity in convex geometry%
            52A20, %Convex sets in n dimensions (including convex hypersurfaces)%
		49Q22. %optimal transportation%
	\\\indent GA is supported by NSF-DMS 2154402. DL acknowledges the support of the Natural Sciences and Engineering Research Council of Canada and the Department of Mathematical and Statistical Sciences at the University of Alberta.
}

\maketitle
\section{Introduction and Main Results}
Let $\gamma$ denote the standard Gaussian probability measure on $\R^n$, $\d \gamma (x) \propto e^{- \frac{1}{2}\vert x \vert^{2}} \d x,$ where $\vert \cdot \vert$ denotes the Euclidean norm. It was shown by Eskenazis and Moschidis \cite{EskenazisMoschidis21} that, if $K_{0}$ and $K_{1}$ are origin-symmetric convex bodies, then
\begin{equation} \label{eq: GaussiandimBM}
\gamma \left( (1-t) K_{0} + t K_{1} \right)^{\frac{1}{n}} \geq (1-t) \gamma (K_{0})^{\frac{1}{n}} + t \gamma (K_{1})^{\frac{1}{n}}.  
\end{equation}
Here $(1-t) K_{0} + t K_{1} = \{ (1-t) x_{0} + t x_{1} : x_{0} \in K_{0} , x_{1} \in K_{1} \}$ denotes the collection of $t$-midpoints of all segments from $K_{0}$ to $K_{1}$. Observe that the inequality \eqref{eq: GaussiandimBM} cannot hold for all compact sets $K_{0}, K_{1}$. This can easily be seen by fixing a set $K_{0}$ of positive Gaussian measure, $K_{1} = \{ x \}$, and sending $x \to \infty$. Without extra conditions on $K_{0}$ and $K_{1}$, the Gaussian measure only satisfies 
\begin{equation}
    \gamma \left( (1-t) K_{0} + t K_{1} \right) \geq \gamma (K_{0})^{1-t} \gamma (K_{1})^{t},
\end{equation}
by virtue of being log-concave. Recall that a measure $\nu$ is said to be \emph{log-concave} if it has a density of the form $\frac{\d \nu}{\d x} = e^{-V}$, $V$ convex, with respect to the Lebesgue measure.

The inequality \eqref{eq: GaussiandimBM} was conjectured by Gardner and Zvavitch \cite{GardnerZvavitch10}, originally for convex bodies containing the origin. But soon after, Nayar and Tkocz \cite{NayarTkocz13} found a counterexample and suggested the assumption of symmetry about the origin. It must be mentioned that the work of Eskenazis and Moschidis closed the proof of \eqref{eq: GaussiandimBM} by verifying a sufficient condition introduced by Kolesnikov and Livshyts \cite{KolesnikovLivshyts21}, which is itself based on a machinery developed by Kolesnikov and E. Milman \cite{KolesnikovMilman17, KolesnikovMilman18}.

More generally, the following conjecture has garnered a lot of attention in the last few years. 
\begin{con} \label{con: DimBM}
Let $\nu$ be an even log-concave measure on $\RR^{n}$.
Then, for origin-symmetric convex bodies $K_{0},K_{1} \subseteq\RR^{n}$,
we have 
\begin{equation} \label{eq: DimBM}
\nu\left((1-t)K_{0}+t K_{1}\right)^{\frac{1}{n}}\ge(1-t)\nu(K_{0})^{\frac{1}{n}}+ t\nu(K_{1})^{\frac{1}{n}}.
\end{equation}
\end{con}
One reason why Conjecture \ref{con: DimBM} is of substantial interest is that it follows from the celebrated \emph{log-Brunn--Minkowski conjecture} of B\"or\"oczky, Lutwak, Yang, and Zhang \cite{BoroczkyLutwakYangZhang12}. This implication was shown by Livshyts, Marsiglietti, Nayar, and Zvavitch \cite{LivshytsMarsigliettiNayarZvavitch17}. Exciting recent developments include works by Livshyts \cite{Livshyts23}, Cordero-Erausquin and Rotem \cite{Cordero-ErasquinRotem23}.

Recently, Aishwarya and Rotem \cite{AishwaryaRotem23} took a completely different route to prove dimensional inequalities such as in \eqref{eq: DimBM}, using entropy. 
\begin{definition}
    Let $\nu$ be a $\sigma$-additive Borel measure on $\R^n$. For a probability measure $\mu$, we define the \emph{relative entropy of $\mu$ with respect to $\nu$} by 
    \begin{equation}
D(\mu \Vert \nu) =
\begin{cases}
\int \left( \frac{\d \mu}{\d \nu} \right) \log \left( \frac{\d \mu}{\d \nu} \right) \d \nu, & \textnormal{ if } \mu \textnormal{ has density w.r.t. } \nu , \\
+ \infty, & \textnormal{ otherwise. }  \\
\end{cases}
\end{equation}
\end{definition}
\begin{notation}
    The relative entropy $D(\mu \Vert \nu)$ is also written as $D(X \Vert Y)$ when $\nu$ is a probability measure, and $X, Y$ are $\R^n$-valued random vectors with distributions $\mu, \nu$, respectively. Note that the joint distribution $(X,Y)$ is not specified because it is immaterial for this definition. See also Definition \ref{def: randomvectors}.
\end{notation}
The technique in \cite{AishwaryaRotem23} is based on the variational principle \cite[Lemma 2.7]{AishwaryaRotem23}:
\begin{equation} \label{eq: measuremaxent}
    \nu (K) = \sup_{\mu \in \mathcal{P}(K)} e^{- D(\mu \Vert \nu )},
\end{equation}
which holds for every compact set $K$ and is attained by the normalised restriction $\nu_{K}$ of $\nu$ to $K$, that is, $\nu_{K}(E) = \frac{\nu(E \cap K)}{\nu(K)}$ for every Borel set $E$. As in formula \eqref{eq: measuremaxent}, we will consistently write $\mathcal{P}(K)$ for the collection of all probability measures on a given set $K$. 

Suppose $\nu$ is a probability measure, and $Y$ is a random vector with distribution $\nu$. In light of the above variational formula, to prove the inequality \eqref{eq: DimBM}, it suffices to show the existence of a joint distribution $(X_{0} , X_{1})$ with the marginals $X_{0},X_{1}$ having distributions $\nu_{K_{0}} , \nu_{K_{1}}$, respectively, such that the following entropy inequality holds: 
\begin{equation} \label{eq: dimBMent}
    e^{ - \frac{1}{n}D ((1-t)X_{0} + t X_{1} \Vert Y)} \geq (1-t) e^{ - \frac{1}{n}D (X_{0} \Vert Y)} + t e^{ - \frac{1}{n}D ( X_{1} \Vert Y)}.  
\end{equation}
This is because the distribution of $(1-t)X_{0} + t X_{1}$ lies in $\mathcal{P} \left( (1-t) K_{0} + t K_{1} \right)$. 

The definition and remarks below clarify our use of some standard terminology regarding joint distributions of random vectors.

\begin{definition} \label{def: randomvectors}
\begin{enumerate}
\item[] 
  \item  Let $X_{0}, X_{1}$ be $\R^n$-valued random vectors. By a joint distribution with marginals $X_{0}, X_{1}$ we mean an $\R^n \times \R^n$-valued random vector $\bar{X}$ such that $\prob \{ \bar{X} \in E \times \R^n \} = \prob \{ X_{0} \in E \} $ and $\prob \{ \bar{X} \in \R^n \times E' \} = \prob \{ X_{1} \in E' \} $ for Borel sets $E, E'$. Here $\prob$ denotes the measure on the underlying probability space over which our random vectors are defined. Such an $\bar{X}$ is often written simply as $(X_{0}, X_{1})$.

 \item   Likewise, a coupling of $\mu_{0}, \mu_{1} \in \mathcal{P}(\R^n)$ is a $\pi \in \mathcal{P}(\R^n \times \R^n)$ such that $\pi(E \times \R^n) = \mu_{0}(E), \pi(\R^n \times E') = \mu_{1}(E')$ for Borel sets $E,E'$.
 \end{enumerate}
\end{definition}
\begin{rems}
\begin{itemize}
\item[] 
    \item If $X_{i}$ has distribution $\mu_{i}$, that is $\prob \{ X_{i} \in E \} = \mu_{i}(E)$ for Borel sets $E$ and $i=0,1$, then the distribution of every joint distribution $(X_{0}, X_{1})$ is a coupling $\pi$ and vice versa. However, we will sometimes also call $(X_{0}, X_{1})$ a coupling. 
    \item If $(X_{0}, X_{1})$ has distribution $\pi$, then the distribution of the corresponding $(1-t) X_{0} + t X_{1}$ is given by the pushforward measure $\left[ (x,y) \mapsto (1-t) x + t y \right]_{\#} \pi \in \mathcal{P}(\R^n)$.
 \end{itemize}   
\end{rems}
The coupling $(X_{0} , X_{1})$ used in \cite{AishwaryaRotem23} to obtain several results is the so-called \emph{optimal coupling} for the \emph{Monge--Kantorovich problem with quadratic cost}, namely the one that minimises $\E \vert X_{0} - X_{1} \vert^{2}$. For example, \cite[Theorem 1.3]{AishwaryaRotem23} implies that, when $Y=Z$ has standard Gaussian distribution, the inequality \eqref{eq: dimBMent} holds for the optimal coupling with a worse exponent ($\frac{1}{2n}$ instead of $\frac{1}{n}$) but for a larger class ($K_{0}, K_{1}$ are only assumed to be \emph{star-shaped} with respect to the origin, not necessarily symmetric or convex). This was the first time that a dimensional Brunn--Minkowski inequality was obtained for the Gaussian measure without convexity assumptions on the admissible sets (which is not possible with the earlier approach). However, while trying to obtain an inequality of the form \eqref{eq: dimBMent} that would strengthen the result of Eskenazis and Moschidis, the authors in \cite{AishwaryaRotem23} faced a very interesting problem.
\begin{qstn} \label{qstn: poincare} \cite{AishwaryaRotem23}
    Suppose $X_{0}, X_{1}$ are $\R^n$-valued random vectors with even strongly log-concave distributions, and assume that $(X_{0} , X_{1})$ have the optimal coupling. Is it true that each $X_{t} = (1-t) X_{0} + t X_{1}$, $t\in (0,1)$, satisfies the Poincar\'e inequality for odd functions with constant $1$?
\end{qstn}
Recall that a random vector $X$ is said to satisfy a Poincar\'e inequality with constant $1$ over a class of functions $\mathcal{F}$, if for every function $f \in \mathcal{F}$, we have $\var (f(X)) \leq \E \vert \grad f (X) \vert^{2}$. Further, a \emph{strongly log-concave} random vector is one with distribution $\mu$ such that $\frac{\d \mu}{\d \gamma}$ is a log-concave function (in this case, $\mu$ is said to be a strongly log-concave measure). The relevance of this property in our context stems from the fact that $\gamma_{K}$ is strongly log-concave whenever $K$ is a convex body. \cite[Theorem 4.5]{AishwaryaRotem23} shows that the desired inequality \eqref{eq: dimBMent} for $Y=Z$ holds for the optimal coupling, and $X_{0}, X_{1}$ even strongly log-concave, if the answer to Question \ref{qstn: poincare} is positive.

The first main result of the present work is that there exists a coupling of even strongly log-concave random vectors such that \eqref{eq: dimBMent} holds when $Y=Z$.
\begin{thm} \label{thm: mainEnt}
    Let $X_{0}, X_{1}$ be $\R^n$-valued random vectors with even strongly log-concave distributions. Then, there is a coupling $(X_{0} , X_{1})$ of $X_{0}$ and $X_{1}$ such that
    \begin{equation} \label{eq: dimBMentZ}
    e^{ - \frac{1}{n}D ((1-t)X_{0} + t X_{1} \Vert Z)} \geq (1-t) e^{ - \frac{1}{n}D (X_{0} \Vert Z)} + t e^{ - \frac{1}{n}D ( X_{1} \Vert Z)}.  
\end{equation}
Moreover, for this coupling, we have equality if and only if $X_{0}$ and $X_{1}$ have the same distribution. 
\end{thm}
\begin{rem} \label{rem: curvaturemain}
    The proof establishes the stronger inequality,
    \begin{equation} \label{eq: globalcd2n}
    e^{-\frac{1}{n} D (X_{t} \Vert Z)} \geq \sigma^{(1-t)} \left( \theta \right) e^{-\frac{1}{n} D (X_{0} \Vert Z)} + \sigma^{(t)} \left( \theta \right)  e^{-\frac{1}{n} D (X_{1} \Vert Z)},
\end{equation}
where $\theta = \left( \E \vert X_{0} - X_{1} \vert^{2} \right)^{\frac{1}{2}}$, and
 \begin{equation}
\sigma^{(t)}(\theta) = 
    \frac{\sin \left( \sqrt{\frac{2}{n}} t \theta \right)}{\sin \left( \sqrt{\frac{2}{n}}  \theta \right)},
  \end{equation}
  for this $\theta$, and $t \in [0,1]$. As discussed in the proof of Theorem \ref{thm: mainEnt}, the value $\theta$ of interest is always strictly less than $\sqrt{n/2} \pi$. The equality characterisation in Theorem \ref{thm: mainEnt} follows from the equality characterisation for $\sigma^{(t)} (\theta) = t$.
\end{rem}
The coupling we use is not the optimal coupling, but nonetheless arises from optimal transport. Let $U, V$ be $\R^n$-valued random vectors satisfying $\E \vert U \vert^{2}, \E \vert V \vert^{2} < \infty$, such that their distributions have density with respect to the Lebesgue measure. Then, a theorem of Brenier (see \cite[Theorem 2.12 (ii)]{Villani03}) guarantees that a unique coupling minimises $\E \vert U - V \vert^{2}$, and furthermore, it is given by $(U , T(U))$ where $T= \grad \phi$ is the gradient of a convex function $\phi$. Note that the map $T$,  called the \emph{Brenier map} from $U$ to $V$, pushes forward the distribution of $U$ to the distribution of $V$. In the present work, we consider the Brenier map $T_{0}$ from $Z$ to $X_{0}$, the Brenier map $T_{1}$ from $Z$ to $X_{1}$, and work with the joint distribution $(X_{0},X_{1}) = (T_{0}(Z) , T_{1}(Z))$.

The contraction theorem of Caffarelli \cite{Caffarelli00} tells us that the Brenier map from the standard Gaussian to any strongly log-concave random vector is $1$-Lipschitz. This automatically gives us that the $X_{t} = (1-t) X_{0} + t X_{1}$ we consider in this paper is a $1$-Lipschitz image of $Z$ under $T_{t} = (1-t)T_{0} + t T_{1}$. Given that $Z$ satisfies a Poincar\'e inequality with constant $1$, a standard change of variables argument immediately shows that $X_{t}$ satisfies the Poincar\'e inequality with constant $1$ for all functions. However, interestingly, we do not use this fact directly. Instead, we use the $1$-Lipschitz property of $T_{t}$ and the Poincar\'e constant of $Z$ separately. It remains an open question whether the optimal coupling also satisfies the conclusion of Theorem \ref{thm: mainEnt}. 

An important feature of the interpolation $X_{t}$, if considered under optimal coupling, is that the trajectories $\{ T_{t} (x) \}_{t \in (0,1)}$ do not cross (in an almost-everywhere sense), and hence the distribution $\mu_{t}$ of $X_{t}$ can be described as the flow of $\mu_{0}$ under a time-dependent velocity field. Yet another useful property under optimal coupling is that the velocity field generated is a gradient field. Both these properties are used in \cite{AishwaryaRotem23}. 

In our case, for $X_{t}$ that we consider, we are not guaranteed the existence of a driving velocity field, nor do we see a reason for this velocity field to be a gradient field even if it exists. The former technical difficulty is overcome by a ``\emph{trajectories do not cross}'' result when $X_{0}, X_{1}$ are ``nice'' (Proposition \ref{prop: velocityexists}), and approximation. The latter issue most prominently appears in the proof of Theorem \ref{thm: mainEnt}, where an inequality such as $\E \tr [\grad v(X_{t})^{2}]  \geq \E \vert v(X_{t}) \vert^{2}$ is needed for a particular odd vector field $v$. This is always true when $v$ is a gradient field and the even random vector $X_{t}$ has Poincar\'e constant $1$, but not in general. To resolve this problem, we explicitly use the structure of the given vector field $v$ (which depends on $T_{t}$) and the Gaussian Poincar\'e inequality. This makes it unclear if our proof would go through if $T_{0}$ and $T_{1}$ were contractions (via the reverse Ornstein--Uhlenbeck process) introduced by Kim and E. Milman \cite{KimMilman12}, and not Brenier maps. Readers familiar with the work of Alesker, Gilboa, and V. Milman \cite{AleskerDarMilman99} may find it intriguing to compare the fact that the coupling used in this paper admits Theorem \ref{thm: mainEnt} (while for other aforementioned couplings such a result is yet unestablished), with Gromov's observation that $\grad \phi [\R^n] + \grad \psi [\R^n] = (\grad \phi + \grad \psi )[\R^n]$ when $\phi, \psi$ are $C^{2}$ convex functions with strictly positive Hessian \cite[1.3.A.]{Gromov90} (see also, \cite[Proposition 2.2]{AleskerDarMilman99}).

As an immediate corollary to Theorem \ref{thm: mainEnt}, using the variational principle \eqref{eq: measuremaxent}, we obtain a new proof of Eskenazis and Moschidis' result. 
\begin{cor} \label{cor: EMBM}
    The dimensional Brunn--Minkowski inequality for the Gaussian measure \eqref{eq: GaussiandimBM} holds if $K_{0}$ and $K_{1}$ are origin-symmetric convex bodies.
\end{cor}
\begin{rem}
In view of Remark \ref{rem: curvaturemain},  it is an interesting question if one can meaningfully bound $ \E \vert X_{0} - X_{1} \vert^{2}$ from below, when $X_{0}, X_{1}$ have distributions $\gamma_{K_{0}}, \gamma_{K_{1}}$, respectively, for symmetric convex bodies $K_{0}, K_{1}$. This could potentially lead to a Gaussian dimensional Brunn--Minkowski inequality for symmetric convex bodies which also incorporates the curvature aspects of the Gaussian measure. As far as we know, this has not been done.
\end{rem}
A fundamental problem in this area concerns obtaining functional forms of geometric inequalities. This means that, given a geometric inequality, one wants to find a functional inequality which recovers the given geometric inequality when applied to functions canonically associated with the involved sets (for example, to indicator functions). For a prototypical example, consider the Brunn--Minkowski inequality in its geometric-mean form that all log-concave measures $\nu$ are known to satisfy:
\begin{equation} \label{eq: lcbm}
    \nu ((1-t) K_{0} + t K_{1}) \geq \nu (K_{0})^{1-t} \nu (K_{1})^{t},
\end{equation}
whenever $K_{0}$ and $K_{1}$ are compact sets in $\R^n$. The functional form of \eqref{eq: lcbm} is the \emph{Pr\'ekopa--Leindler inequality} which concludes 
\begin{equation} \label{eq: lcpl}
    \int h \d \nu \geq \left( \int f \d \nu \right)^{1-t} \left( \int g \d \nu \right)^{t}, 
\end{equation}
whenever $f,g,h$ are non-negative functions satisfying
\begin{equation}
    h((1-t)x + t y ) \geq f(x)^{1-t}  g (y)^{t}, 
\end{equation}
for all $x,y$. Of course, if $f$ and $g$ are indicator functions of $K_{0}$ and $K_{1}$, respectively, then the indicator of $(1-t) K_{0} + t K_{1}$ is an admissible choice for $h$, thereby producing \eqref{eq: lcbm}. While several proofs of the Pr\'ekopa--Leindler inequality exist (for example, see \cite{Gardner02}), an elegant proof can be obtained from the entropy form of \eqref{eq: lcbm}: every pair of $\R^n$-valued random vectors $X_{0}, X_{1}$ with density (with respect to the Lebesgue measure) have a joint distribution $(X_{0} , X_{1})$ such that
\begin{equation} \label{eq: lcdispcon}
    D((1-t)X_{0} + t X_{1} \Vert Y) \leq (1-t) D(X_{0} \Vert Y) + t D( X_{1} \Vert Y),
\end{equation}
where $Y$ is a random vector with distribution $\nu$. The fact that the optimal coupling $(X_{0}, X_{1})$ satisfies \eqref{eq: lcdispcon} (see \cite[Theorem 9.4.11]{AmbrosioGigliSavare08}) is well known, and often recorded as the ``\emph{displacement convexity of entropy} on the metric measure space $(\R^n , \vert \cdot \vert , \nu)$''. To go from \eqref{eq: lcdispcon} to \eqref{eq: lcpl} one can use the Donsker--Varadhan duality formula \cite[Section 2]{DOnskerVaradhan83} describing the Legendre transform of relative entropy. It says, for $\nu$-integrable functions $\phi$, we have
\begin{equation} \label{eq: donskervaradhan}
    \log \int e^{\phi} \d \nu = \sup_{\mu \ll \nu } \left[ \int \phi \d \mu - D(\mu \Vert \nu)  \right],
\end{equation}
where the supremum on the right is over all probability measures $\mu$ absolutely continuous with respect to $\nu$, and equality is attained in \eqref{eq: donskervaradhan} for $\d \mu \propto e^{\phi} \d \nu$. 

We do not spell out the details of the implication \eqref{eq: lcdispcon} $\Rightarrow$ \eqref{eq: lcpl} because the reader may infer the general idea from our proof of Theorem \ref{thm: mainbbl}, which is rather short. Nonetheless, it is apt to remark here that this technique stands out because it entirely operates at the level of integrals and does not appeal to local estimates on the integrands (other than the one granted by assumption), thereby making it possible to extract functional inequalities even if a convexity property of entropy is only available on a restricted class of measures. The same cannot be said about some other transport-based proofs of the Pr\'ekopa--Leindler inequality (or its generalisations). Besides, this method works in measure spaces without any smooth structure. For example, the reader may find beautiful applications to discrete structures in works of Gozlan, Roberto, Samson, and Tetali \cite{GozlanRobertoSamsonTetali21}, and Slomka \cite{Slomka24}.

We will use this duality to obtain a functional form of the dimensional Brunn--Minkowski inequality \eqref{eq: GaussiandimBM}, which is our second main result. 
\begin{notation}
    Let $x,y \geq 0$. We write 
   \begin{equation}
   M_p^t (x , y ) \coloneqq \begin{cases}
        \left( (1-t) x^{p} + t y^{p} \right)^{\frac 1 p}, &\hbox{ for } xy > 0,
            \\
            0 &\hbox{ otherwise,}
    \end{cases}
 \end{equation}
for $t \in [0,1]$ and $p \in (- \infty , 0) \cup (0 , \infty)$. This is extended to $p \in \{ - \infty , 0 , \infty \}$ by taking limits. Thus, in particular, $M_{-\infty}^{t} (x , y ) = \min \{ x , y \}, M_{0}^{t} (x , y ) = x^{1-t} y^{t}, M_{\infty}^{t} (x , y ) = \max \{ x , y \},$ if $xy > 0$.
\end{notation}
\begin{thm} \label{thm: mainbbl}
    Let $p \geq 0$, and suppose $f, g, h$ are non-negative functions on $\R^n$ with $f,g$ even log-concave and $\gamma$-integrable, such that 
    \begin{equation}
        h((1-t)x_{0} + t x_{1}) \geq M_{p}^{t} \left( f(x_{0}) , g(x_{1}) \right).
    \end{equation}
    Then, we have 
    \begin{equation}
         \int h \d \gamma  \geq M_{\frac{p}{1 + np}}^{t} \left( \int f \d \gamma, \int g \d \gamma \right). 
    \end{equation}
\end{thm}
Indeed, when $f$ and $g$ are indicators of symmetric convex bodies, and $p= \infty$, one recovers \eqref{eq: GaussiandimBM}. Inequalities such as in the theorem above are sometimes called \emph{Borell--Brascamp--Lieb inequalities} after the works by Borell \cite{Borell75}, and Brascamp--Lieb \cite{BrascampLieb76BBL}.

To the best of our knowledge, the argument for obtaining Theorem \ref{thm: mainbbl} from Theorem \ref{thm: mainEnt}, though simple, is new. Previously, it was not clear how to apply duality \eqref{eq: donskervaradhan} to inequalities such as \eqref{eq: dimBMentZ} that are not linear in relative entropy. Exactly the same idea as we use in the proof of Theorem \ref{thm: mainbbl} gives further dimension-dependent functional inequalities for functions that are not necessarily log-concave, as discussed below. 

Recall that a convex function $V: \R^n \to \R$ is said to be $\beta$-homogeneous if $V( \lambda x) = \lambda^{\beta} V(x)$, for every $x \in \R^n$ and $\lambda > 0$. Consider the probability measure $\nu \propto e^{-V} \d x$, such that $V$ is $\beta$-homogeneous for some $\beta \in (1, \infty)$. Say $\nu$ is represented by a random vector $Y$. Then, \cite[Theorem 1.4]{AishwaryaRotem23} states that, for random vectors $X_{0}$ and $X_{1}$ having radially decreasing density with respect to $\nu$, there exists a coupling $(X_{0}, X_{1})$ (in fact, the optimal coupling works) such that 
\begin{equation} \label{eq: convexentother}
    e^{- \frac{\beta - 1}{\beta n} D((1-t)X_{0} + t X_{1} \Vert Y)} \geq (1-t)     e^{- \frac{\beta - 1}{\beta n} D(X_{0} \Vert Y)} + t     e^{- \frac{\beta - 1}{\beta n} D(X_{1} \Vert Y)}.
\end{equation}
From this, we have the following result.
\begin{thm} \label{thm: functionalformold}
    Consider the probability measure $ \d \nu \propto e^{-V} \d x$, such that $V$ is $\beta$-homogeneous for some $\beta \in (1, \infty)$. Let $p \geq 0$, and suppose $f,g,h$ are non-negative functions on $\R^n$ with $f,g$ radially decreasing and $\nu$-integrable, such that
     \begin{equation}
        h((1-t)x_{0} + t x_{1}) \geq M_{p}^{t} \left( f(x_{0}) , g(x_{1}) \right).
    \end{equation}
    Then, we have 
    \begin{equation}
         \int h \d \nu  \geq M_{\frac{(\beta - 1)p }{(\beta - 1) + \beta n p}}^{t} \left( \int f \d \nu, \int g \d \nu \right). 
    \end{equation}
\end{thm}
The standard Gaussian measure falls under the regime $\beta=2$. Thus, Theorem \ref{thm: functionalformold} can be applied to a larger class of functions compared to Theorem \ref{thm: mainbbl}, but at the same time draws a weaker conclusion.

\subsection{Related work}
In a previous version of our paper, we had the privilege of reporting two independent works on functional forms of the Eskenazis--Moschidis inequality that were not in the public domain at the time. Both works are now available online. The work by Dario Cordero-Erausquin and Alexandros Eskenazis \cite{Cordero-ErausquinEskenazis25} is based on the so-called $L^{2}$-method. Beautifully aligning with the viewpoint of classical convexity, \cite[Theorem 1]{Cordero-ErausquinEskenazis25} presents a concavity principle which is equivalent to our Theorem \ref{thm: mainbbl} under greater restrictions on the involved functions but applies to a wider class of reference measures (as in \cite{Cordero-ErasquinRotem23}). On the other hand, the work of Andreas Malliaris, James Melbourne, Cyril Roberto, and Michael Roysdon \cite{MalliarisMelbourneRobertoRoysdon25} contains an elegant measure-theoretic technique that produces very general functional inequalities directly from geometric inequalities having a certain form. For example, starting with the Eskenazis--Moschidis inequality \eqref{eq: GaussiandimBM}, \cite[Theorem 1.1]{MalliarisMelbourneRobertoRoysdon25} significantly generalises Theorem \ref{thm: mainbbl} to admit all even unimodal functions and parameters $p \geq -1/n$. Likewise, a generalisation of Theorem \ref{thm: functionalformold} can be found in \cite[Theorem 3.2]{MalliarisMelbourneRobertoRoysdon25}.

\subsection{Acknowledgements}
We are grateful to Liran Rotem for his continued and generous sharing of insights on the topic of dimensional Brunn--Minkowski inequalities, and to Alexandros Eskenazis, for sharing his joint results as well as his comments (in particular, but not limited to, his suggestion to discuss equality cases in Theorem \ref{thm: mainEnt}). Many thanks to James Melbourne for kindly sharing his work, and to Alexander Volberg for enriching discussions on a related problem. We sincerely acknowledge the helpful discussions with Galyna Livshyts and Emma Pollard on potential functional forms of the Eskenazis--Moschidis inequality. We also thank the anonymous reviewer for their valuable comments. 

\subsection{Organisation of the paper}
The proof of Theorem \ref{thm: mainEnt} is based on an Eulerian description of mass transport. The required background is presented in Section \ref{sec: prelim}. Proofs of Theorem \ref{thm: mainEnt} and Theorem \ref{thm: mainbbl} appear in Section \ref{sec: mainproofs}. Theorem \ref{thm: functionalformold} follows along the same lines as Theorem \ref{thm: mainbbl}, hence we omit its proof. 
\subsection*{Declarations of interest} 
On behalf of all authors, the corresponding author states that there is no conflict of interest.
\subsection*{Data availability statement}
Our manuscript has no associated data.

\section{Preliminaries} \label{sec: prelim}
First, we note that $\frac{\d^{2}}{\d t^{2}} e^{- \frac{1}{n} D(X_{t} \Vert Z)} \leq 0$ is equivalent to $\frac{\d^{2}}{ \d t^{2}} D(X_{t} \Vert Z) \geq \frac{1}{n} \left( \frac{\d}{ \d t} D(X_{t} \Vert Z) \right)^{2}$, whenever the relevant quantities have the required regularity. Thus, we would like to compute $\frac{\d^{2}}{ \d t^{2}} D(X_{t} \Vert Z)$ and $\frac{\d}{ \d t} D(X_{t} \Vert Z)$. Such local computations are often best performed in the language of velocity fields.

Suppose a curve $\{ \mu_{t} \}_{t \in [0,1]}$ of probability measures on $\R^n$ is given. A time-dependent velocity field $v_{t}$ is said to be \emph{compatible} with $\{ \mu_{t} \}_{t \in [0,1]}$ if 
\begin{equation} \label{eq: conteq}
\partial_{t} \mu_{t} + \div (v_{t} \mu_{t}) = 0
\end{equation}
is satisfied in the weak sense, where $\div$ denotes divergence. The latter equation means that 
\begin{equation} \label{eq: conteqweak}
 \frac{\d }{\d t} \int f \d \mu_{t} = \int \langle \grad f , v_{t} \rangle \d \mu_{t},
\end{equation}
for all compactly supported smooth functions $f$. Once Equation \eqref{eq: conteq} is known to hold, Equation \eqref{eq: conteqweak} holds under wider generality; for example, it holds for all bounded Lipschitz smooth functions (see \cite[Chapter 8]{AmbrosioGigliSavare08}). Moreover, under additional regularity assumptions on $v_{t}$, the class of functions $f$ for which \eqref{eq: conteqweak} holds can be expanded further. For instance, in the setup of \cite[Example 2.1]{AishwaryaLi25KP}, the elementary calculation therein shows that \eqref{eq: conteqweak} holds for all bounded Lipschitz functions. 

Ignoring all regularity issues, we compute the derivatives of $D(\mu_{t} \Vert \gamma)$ twice when a compatible velocity field is given. We write the result for a general log-concave measure $\nu$ since it may be of independent interest.
\begin{prop} \label{prop: derivativesofentropy}
    Let $\d \nu = e^{-W} \d x$, for smooth convex $W$. Consider a curve of probability measures $\{ \mu_{t} \}_{t \in [0,1]}$ with a compatible velocity field $v_{t}$. Suppose $v_{t}$ is sufficiently smooth, then 
    \begin{equation}
\begin{split}
    \frac{\d}{\d t} D(\mu_{t} \Vert \nu) = - \int \div^{W} (v_{t}) \d \mu_{t}, \\
\end{split}
    \end{equation}
    where $\div^{W}(v) = \div (v) - \langle \grad W , v \rangle$, for vector fields $v$.
    Moreover, if the trajectories of $v_{t}$ take each particle along a straight line with constant speed, then 
    \begin{equation}
            \frac{\d^{2}}{\d t^{2}}  D(\mu_{t} \Vert \nu) = \int \mathcal{G}^{W} (v_{t}) \d \mu_{t}, \\
    \end{equation}
    where $\mathcal{G}^{W} (v) = \tr (\grad v)^{2} + \langle \grad^{2} W \cdot v , v \rangle$, for vector fields $v$.
\end{prop}
\begin{rem} \label{rem: whatisstraightline}
    Consider the one-parameter family of maps $S_{t}$ which solves $\frac{\d}{\d t} S_{t} (x) = v_{t} (S_{t}(x))$, whose existence is guaranteed by the smoothness assumptions on $v_{t}$. Then, by ``\textit{the trajectories of $v_{t}$ take each particle along a straight line with constant speed}'' we mean that the curve $t \mapsto S_{t}(x)$, for each $x$, is a straight line with constant speed. Thus, $\frac{\d^{2}}{\d t^{2}} S_{t} (x) = 0$, and consequently, $\partial_{t} v_{t} + \grad_{v_{t}} v_{t} = 0$.
\end{rem}
\begin{proof}
    Let $ \rho_{t}$ denote the density of $\mu_{t}$ with respect to $\nu$. Then,
    \begin{equation} \label{eq: firstderivativeentropycomp}
\begin{split}
\frac{\d}{\d t} D(\mu_{t} \Vert \nu) & = \frac{\d}{\d t} \int  \log \rho_{t} \d \mu_{t} = \int \partial_{t}  \log \rho_{t} \d \mu_{t} + \int \langle \grad \log \rho_{t} ,  v_{t} \rangle \d \mu_{t} \\
&=  \frac{\d}{\d t}\int  \rho_{t} \d \nu + \int \langle \grad \log \rho_{t} , v_{t} \rangle \rho_{t} \d \nu =  \int \langle \grad \rho_{t} , v_{t} \rangle \d \nu \\
&= \int \langle \grad \rho_{t} , v_{t} \rangle e^{-W} \d x = - \int \rho_{t} \div (e^{-W}v_{t}) \d x \\
&= - \int \rho_{t} \, \div^{W} (v_{t}) \d \nu = - \int\div^{W} (v_{t}) \d \mu_{t}. 
\end{split}
\end{equation}
In the above computation, the second equality uses the chain rule and the continuity equation \eqref{eq: conteq}, and the sixth equality is an application of integration by parts.

Further, as noted in Remark \ref{rem: whatisstraightline}, we have $\partial_{t}v_{t} + \grad_{v_{t}} v_{t} = 0$ if the trajectories of $v_{t}$ take particles along a straight line with constant speed. This allows the following computation to proceed.   
\begin{equation}
\begin{split}
\frac{\d^{2}}{\d t ^{2}}  D(\mu_{t} \Vert \nu) &= -  \frac{\d }{\d t} \int  \div^{W} (v_{t}) \d \mu_{t} = - \int  \div^{W} \left( \del_{t} v_{t} \right) \d \mu_{t} - \int \langle \grad \div^{W} (v_{t}) , v_{t} \rangle  \d \mu_{t}  \\
&= \int  \div^{W} \left( \grad_{v_{t}} v_{t} \right)  \d \mu_{t} - \int \langle \grad \div^{W} (v_{t}) , v_{t} \rangle  \d \mu_{t} = \int \mathcal{G}^{W}(v_{t})  \d \mu_{t},
\end{split}
\end{equation}
where in the second equality we use the chain rule and the continuity equation \eqref{eq: conteq}, while the last equality uses the pointwise formula
\begin{equation} \label{eq: weightedbochner}
\mathcal{G}^{W} (v) =  \div^{W} \left( \grad_{v} v \right) -   \langle \grad \div^{W} (v) , v \rangle, 
\end{equation}
which holds for smooth $v$. Formula \eqref{eq: weightedbochner} is an easily obtainable weighted-version of the Bochner formula for vector fields (see, for example, \cite[Equation 14.26]{Villani09}).
\end{proof}

To utilise the above formulas for the derivatives of entropy, we need to establish the existence of compatible velocity fields in the cases of interest. We let $I_{n \times n}$ denote the $n \times n$ identity matrix.
\begin{prop} \label{prop: velocityexists}
Fix a probability measure $\d \nu = e^{-W} \d x$, and maps $T_{0} = \grad \phi_{0} , T_{1} = \grad \phi_{1} : \R^n \to \R^n$, where $\phi_{0}, \phi_{1}$ are convex functions such that $\vert \grad \phi_{1} - \grad \phi_{0} \vert \in L^{1}(\nu)$. Denote by $\mu_{t} = {T_{t}}_{\#} \nu$, where $T_{t} = (1-t) T_{0} + t T_{1}$. Suppose $\grad^{2} \phi_{0}$ and $\grad^{2} \phi_{1}$ are both lower-bounded (in the positive semidefinite order) by $\lambda I_{n \times n}$ for some $\lambda > 0$. Then, the equation 
\begin{equation}
    v_{t} (T_{t}(x)) = \frac{\d}{\d t}T_{t} (x)
\end{equation}
defines a velocity field compatible with the curve $\{ \mu_{t} \}_{t \in [0,1]}$.
\end{prop}
\begin{proof}
    Evidently, the only obstruction in defining a velocity field is that two trajectories $T_{t}(x)$ and $T_{t}(y)$ cross each other at some time $t \in (0,1)$, that is, if there is a $t_{\star}$ such that $T_{t_{\star}} (x) = T_{t_{\star}} (y)$ for $x \neq y$. However, 
    \begin{equation}
        \begin{split}
            \langle T_{t} (x) - T_{t} (y) , x-y \rangle &= \langle (1-t) \left( T_{0}(x) - T_{0}(y) \right) +  t \left( T_{1}(x) - T_{1}(y)  \right) , x- y \rangle \\
            &= (1-t) \langle   T_{0}(x) - T_{0}(y)  , x- y \rangle + t \langle   T_{1}(x) - T_{1}(y)  , x- y \rangle \\
            & \geq (1-t)\lambda \vert x - y \vert^{2} + t \lambda \vert x - y \vert^{2} = \lambda \vert x - y \vert^{2},
        \end{split}
    \end{equation}
    because $\phi_{0} - \frac{\lambda}{2}\vert x \vert^{2}$ and $\phi_{1} - \frac{\lambda}{2}\vert x \vert^{2}$ are convex and consequently have monotone gradients. Thus, the possibility of this obstruction is ruled out. Now we verify the compatibility, where the $\nu$-integrability of $\vert \grad \phi_{1} - \grad \phi_{0} \vert$ is used. For a compactly supported smooth function $f$, essentially the same calculation and reasoning as in \cite[Example 2.1]{AishwaryaLi25KP} can be repeated to get
    \begin{equation}
    \begin{split}
        \frac{\d}{\d t} \int f \d \mu_{t} &= \frac{\d}{\d t} \int f (T_{t}(x)) \d \nu (x)  = \int \langle \grad f (T_{t} (x)) , \frac{\d}{\d t}T_{t} (x) \rangle \d \nu  (x) \\
        &= \int \langle \grad f (T_{t}(x)) , v_{t} (T_{t} (x)) \rangle \d \nu = \int \langle \grad f , v_{t} \rangle \d \mu_{t}. \qedhere
        \end{split}
    \end{equation}
\end{proof}
\begin{rem}
    As mentioned previously, in this case, \eqref{eq: conteqweak} is valid for bounded Lipschitz functions $f$. The integrability assumption is very mild, for example, it is satisfied when $\mu_{0}$ and $\mu_{1}$ have finite second moments. 
\end{rem}
\section{Proof of the main results} \label{sec: mainproofs}
In this section, we solely work with the Gaussian measure as the reference measure. Thus, $\nu$ from the previous section is taken to be $\gamma$. In this case, we will denote $\div^{W}$ by $\divtil$ and $\mathcal{G}^{W}$ by $\gtil$. We will specify a coupling $(X_{0}, X_{1})$ of even strongly log-concave random vectors and proceed to prove the inequality \eqref{eq: dimBMentZ} for this coupling. In the proof, it will be pointed out how we actually end up proving the stronger inequality \eqref{eq: globalcd2n} from which \eqref{eq: dimBMentZ} and its equality conditions follow, thereby completing the proof of Theorem \ref{thm: mainEnt}.  
\begin{proof}[Proof of Theorem \ref{thm: mainEnt}] 
Suppose $X_{0}, X_{1}$ are even strongly log-concave random vectors in $\R^n$ with distributions $\mu_{0}, \mu_{1}$, respectively. Let $T_{0}$ and $T_{1}$ be Brenier maps from the standard Gaussian $Z$ to $X_{0}$ and $X_{1}$, respectively. With the joint distribution $(T_{0}(Z), T_{1}(Z))$, and $X_{t} = (1-t) X_{0} + t X_{1}$, we want to prove $\frac{\d^{2}}{\d t^{2}} e^{- \frac{1}{n} D(X_{t} \Vert Z)} \leq 0$, that is, $\frac{\d^{2}}{ \d t^{2}} D(X_{t} \Vert Z) \geq \frac{1}{n} \left( \frac{\d}{ \d t} D(X_{t} \Vert Z) \right)^{2}$.

Suppose $\d \mu_{0} \propto e^ {- U_{0}} \d x $ and $\d \mu_{1} \propto e^{ - U_{1}} \d x$. First, we assume that there is a $ \kappa < \infty $ such that $\grad^{2} U_{0} , \grad^{2} U_{1} \leq \kappa I_{n \times n}$. By strong log-concavity, we already have $\grad^{2} U_{0} , \grad^{2} U_{1} \geq I_{n \times n}$. Thus, by Caffarelli's contraction theorem (or a form thereof, see the statement in \cite[Theorem 1]{ChewiPooladian23} or \cite{Kolesnikov11}), we get that $T_{0}, T_{1}$ are both $1$-Lipschitz, while $T^{-1}_{0}, T^{-1}_{1}$ are $\sqrt{\kappa}$-Lipschitz. If we write $T_{0} = \grad \phi_{0}, T_{1} = \grad \phi_{1}$ as gradients of convex functions, then these bounds translate to $\frac{1}{\sqrt{\kappa}} I_{n \times  n} \leq \grad^{2} \phi_{0}, \grad^{2} \phi_{1} \leq I_{n \times n}$. We infer from Proposition \ref{prop: velocityexists} that, if $\mu_{t} = {T_{t}}_{\#} \gamma$, $T_{t} = (1-t) T_{0}  +  t T_{1}$ (thus $\mu_{t}$ is the distribution of $X_{t}$), then a velocity field $v_{t}$ compatible with $\mu_{t}$ is well defined by $v_{t} (T_{t} (x)) = \frac{\d}{\d t} T_{t} (x)$. Further, the smoothness of the velocity field $v_{t}$ required to apply Proposition \ref{prop: derivativesofentropy} can be obtained by Caffarelli's regularity theory \cite{Caffarelli90, Caffarelli92}. Thus, $\frac{\d^{2}}{ \d t^{2}} D(X_{t} \Vert Z) \geq \frac{1}{n} \left( \frac{\d}{ \d t} D(X_{t} \Vert Z) \right)^{2}$ amounts to proving 
\begin{equation}
\int \gtil (v_{t} ) \d \mu_{t} \geq \frac{1}{n} \left( \int \divtil (v_{t}) \d \mu_{t}  \right)^{2}.
\end{equation}
From here we mimic the argument in \cite[Theorem 4.5]{AishwaryaRotem23} with some modifications, but applied to vector fields, where we also use an analogue of a crucial auxiliary construction from \cite{EskenazisMoschidis21}. We will prove the stronger inequality
\begin{equation} \label{eq: strongerlocal}
    \int \gtil (v_{t} ) \d \mu_{t} \geq 2 \int \vert v_{t} \vert^{2} \d \mu_{t} + \frac{1}{n} \left( \int \divtil (v_{t}) \d \mu_{t}  \right)^{2}.
\end{equation}
Let $u_{t}(x) = v_{t} (x) - \frac{l}{n} x $, where $l = \int \divtil (v_{t}) \d \mu_{t}$. Then,
\begin{equation}
    \begin{split}
        \tr (\grad v_{t})^{2} &= \tr \left( \grad u_{t} + \frac{l}{n} I_{n \times n} \right)^{2} = \tr \left( (\grad u_{t})^{2} + \frac{2l}{n} \grad u_{t} + \frac{l^{2}}{ n^{2}} I_{n \times  n}\right) \\
        &= \tr (\grad u_{t})^{2} + \frac{2l}{n} \div ( u_{t} ) + \frac{l^{2}}{n} =  \tr (\grad u_{t})^{2} + \frac{2l}{n} \div ( v_{t} ) - \frac{l^{2}}{n}\\
        &= \tr (\grad u_{t})^{2} + \frac{2l}{n} \left( \divtil (v_{t}) + \langle x , v_{t} \rangle   \right) - \frac{l^{2}}{n} \\
        &=  \tr (\grad u_{t})^{2} + \frac{2l}{n} \langle x , v_{t} \rangle + \left( \frac{2l}{n} \divtil (v_{t}) - \frac{l^{2}}{n} \right). 
    \end{split}
\end{equation}
To continue the proof in the mould of \cite[Theorem 4.5]{AishwaryaRotem23}, we would like to show that $ \int \tr (\grad u_{t})^{2} \d \mu_{t} \geq \sum_{i} \int \vert  u_{t}^{(i)} \vert^{2} \d \mu_{t}$, where we write $u_{t} = (u_{t}^{(1)}, \ldots , u_{t}^{(n)})$ in its components. This step is slightly more involved than in \cite{AishwaryaRotem23} (see Remark \ref{rem: notagradientfield}) and we are forced to take the following route. 

Using the chain rule, one has 
\[\grad u_{t} (T_{t}(x))\grad T_{t}(x) = \grad [u_{t}(T_{t}(x))] = \grad^{2} \phi_{1}(x)-\grad^{2} \phi_{0}(x)-\frac{l}{n} \grad T_{t}(x).\]
Let $A=\grad^{2} \phi_{1}(x)-\grad^{2} \phi_{0}(x)-\frac{l}{n} \grad T_{t}(x)$ and $B=(\grad T_{t}(x))^{-1}$, note that the matrices $A$ and $B$ are symmetric; furthermore, $B \geq I_{n \times n}$. Therefore, we have 

\begin{equation}
\begin{split}
\tr [\grad u_{t}(T_{t}(x))]^{2} &=\tr(AB)^{2}=\tr(ABAB)=\tr(B^{1/2}ABAB^{1/2}) \\
&\geq \tr (B^{1/2}A^{2}B^{1/2})= \tr(A^{2}B)=\tr(ABA) \geq \tr(A^2)\\
&=\tr (\grad [u_{t}(T_{t}(x))])^2, \\
\end{split}
\end{equation}
where both the inequalities follow from the monotonicity of trace under the positive semidefinite ordering.

The trace inequality above gives us the following: 

\begin{equation}
\begin{split}
 \int \tr (\grad u_{t})^{2} \d \mu_{t}
 &= \int \tr [\grad u_{t}(T_{t}(x))]^{2}  \d \gamma \geq  \int \tr (\grad [u_{t}(T_{t}(x))])^2 \d \gamma \\
 &= \sum_{i} \int \vert \grad [u_{t}^{(i)}(T_{t}(x))] \vert^{2} \d \gamma,\\
\end{split}
\end{equation}
where the last equality uses the fact that $u_{t}(T_{t}(x))$ is a gradient field, which can be seen from the expression $u_{t}(T_t(x))=(1-\frac{tl}{n})\grad \phi_{1}(x)-(1+\frac{(1-t)l}{n}) \grad \phi_{0}(x)$. Furthermore, since both $\grad \phi_{0}(x)$ and $\grad \phi_{1}(x)$ are odd, $u_{t}(T_t(x))$ is an odd function of $x$ and thus integrates to $0$ with respect to any even measure.

Applying the Gaussian Poincar\'e inequality to each component of $u_{t}(T_t(x))$, we obtain 
% \begin{equation}
% \begin{split}
%     \int \tr (\grad u_{t})^{2} \d \mu_{t} &= \sum_{i} \int \vert \grad u_{t}^{(i)} \vert^{2} \d \mu_{t} \geq \sum_{i} \int \left( u_{t}^{(i)} \right)^{2} \d \mu_{t} \\ 
%     & = \sum_{t} \int \left( v_{t}^{(i)} - \frac{l}{n} x^{(i)} \right)^{2} \d \mu_{t} = \int \left( \vert v_{t} \vert^{2} - \frac{2l}{n} \langle x , v_{t} \rangle + \frac{l^{2}}{n^{2}} \vert x \vert^{2} \right) \d \mu_{t}.
%     \end{split}
% \end{equation}
\begin{equation}
\begin{split}
    \sum_{i} \int \vert \grad [u_{t}^{(i)}(T_{t}(x))] \vert^{2} \d \gamma
    &\geq  \sum_{i} \int \vert u_{t}^{(i)}(T_{t}(x)) \vert^{2} \d \gamma = \sum_{i} \int \left( u_{t}^{(i)}(x) \right)^{2} \d \mu_{t} \\ 
    & = \sum_{t} \int \left( v_{t}^{(i)}(x) - \frac{l}{n} x^{(i)} \right)^{2} \d \mu_{t} \\
    &= \int \left( \vert v_{t} \vert^{2} - \frac{2l}{n} \langle x , v_{t} \rangle + \frac{l^{2}}{n^{2}} \vert x \vert^{2} \right) \d \mu_{t}.
    \end{split}
\end{equation}
Putting this into the expression for $\tr (\grad v_{t} )^{2}$ from before,
\begin{equation}
    \begin{split}
        \int \gtil (v_{t}) \d \mu_{t} &= \int  \left( \tr ( \grad v_{t})^{2}  
 +  \vert v_{t} \vert^{2} \right) \d \mu_{t} \\
 &= \int  \left( \tr ( \grad u_{t})^{2} + \frac{2l}{n} \langle x , v_{t} \rangle + \left( \frac{2l}{n} \divtil (v_{t}) - \frac{l^{2}}{n} \right)  
 +  \vert v_{t} \vert^{2} \right) \d \mu_{t} \\
 &\geq \int  \left(\vert v_{t} \vert^{2} - \frac{2l}{n} \langle x , v_{t} \rangle + \frac{l^{2}}{n^{2}} \vert x \vert^{2} + \frac{2l}{n} \langle x , v_{t} \rangle + \left( \frac{2l}{n} \divtil (v_{t}) - \frac{l^{2}}{n} \right)  
 +  \vert v_{t} \vert^{2} \right) \d \mu_{t} \\ 
 &= \int  \left( 2 \vert v_{t} \vert^{2}  + \frac{l^{2}}{n^{2}} \vert x \vert^{2}  + \left( \frac{2l}{n} \divtil (v_{t}) - \frac{l^{2}}{n} \right)  
  \right) \d \mu_{t} \\
  & \geq 2 \int    \vert v_{t} \vert^{2}  \d \mu_{t} + \int  \left( \frac{2l}{n} \divtil (v_{t}) - \frac{l^{2}}{n} \right) \d \mu_{t} =   2 \int    \vert v_{t} \vert^{2}  \d \mu_{t} +  \frac{1}{n} l^{2},
    \end{split}
\end{equation}
    where the last equality follows from the definition of $l$. This establishes inequality \eqref{eq: strongerlocal}.

    Note that, 
    \begin{equation}
    \begin{split}
        \int \vert v_{t} \vert^{2} \d \mu_{t} &= \int \vert v_{t} (T_{t}(x)) \vert^{2} \d \gamma (x) = \int \vert \frac{\d}{\d t} T_{t} (x) \vert^{2} \d \gamma (x) \\
        &= \int \vert T_{0} (x) - T_{1} (x) \vert^{2} \d \gamma = \E \vert X_{0} - X_{1} \vert^{2}.
    \end{split}    
\end{equation}
Thus, we have shown that
\begin{equation} \label{eq: localcd2nreg}
    \frac{\d^{2}}{\d t^{2}} D(X_{t} \Vert Z) \geq 2 \E \vert X_{0} - X_{1} \vert^{2} + \frac{1}{n} \left( \frac{\d}{\d t} D(X_{t} \Vert Z) \right)^{2},
\end{equation}
under the assumed regularity on $X_{0},X_{1}$ and the chosen coupling $(X_{0},X_{1})$.
We claim that, 
\begin{equation} \label{eq: globalcd2nreg}
    e^{-\frac{1}{n} D (X_{t} \Vert Z)} \geq \sigma^{(1-t)} \left( \left( \E \vert X_{0} - X_{1} \vert^{2} \right)^{\frac{1}{2}} \right) e^{-\frac{1}{n} D (X_{0} \Vert Z)} + \sigma^{(t)} \left( \left( \E \vert X_{0} - X_{1} \vert^{2} \right)^{\frac{1}{2}} \right)  e^{-\frac{1}{n} D (X_{1} \Vert Z)},
\end{equation}
where 
 \begin{equation}
\sigma^{(t)}(\theta) = 
\begin{cases}
    \frac{\sin \left( \sqrt{\frac{2}{n}} t \theta \right)}{\sin \left( \sqrt{\frac{2}{n}}  \theta \right)}, & \textnormal{if } \theta < \sqrt{\frac{n}{2}} \pi, \\
    \infty, &\textnormal{ otherwise,}\\
\end{cases}
  \end{equation}
  for $t \in [0,1]$. This claim follows from the local version in inequality \eqref{eq: localcd2nreg} and a comparison principle, as applied in \cite[Lemma 5.3]{AishwaryaRotem23} or \cite[Lemma 2.2]{ErbarKuwadaSturm15}. Additionally, keeping in mind the triangle inequality for the Wasserstein metric and \cite[Remark 7]{AishwaryaRotem23},  $\left( \E \vert X_{0} - X_{1} \vert^{2} \right)^{\frac{1}{2}}$ is always less than $\sqrt{n/2} \pi$.

    Since $\sigma^{(t)} \geq t$, we have proved Theorem \ref{thm: mainEnt} (and the claim in Remark \ref{rem: curvaturemain})under the assumption that $\grad^{2} U_{0} , \grad^{2} U_{1} \leq \kappa I_{n \times n} < \infty $. This assumption can be removed via the following approximation argument. 

Let $\epsilon \in (0,1/2)$, and define the Ornstein--Uhlenbeck evolutes $X^{\epsilon}_{i}:=\sqrt{1-\epsilon} X_{i} +\sqrt{\epsilon} Z'$ for $i=0,1$, where $Z'$ is a standard Gaussian random vector independent of both the $X^{\epsilon}_{i}$. Denote by $\mu^{\epsilon}_{i}$ the distribution of $X^{\epsilon}_{i}$, and suppose that $\rho^{\epsilon}_{i}=\frac{\d \mu^{\epsilon}_{i}}{ \d x} = e^{ - U^{\epsilon}_{i}(x)}$ for $i=0,1$. Obviously, the $X^{\epsilon}_{i}$'s  are even random vectors in $\R^n$. Moreover, since the Ornstein--Uhlenbeck process (see, for example, \cite{BakryGentilLedoux14}) preserves strongly log-concavity of measures,  they are also strongly log-concave. This means $\grad^{2} U^{\epsilon}_{i} \geq I_{n \times n}$ for $\forall \epsilon \in (0,1/2)$. Moreover, a direct calculation reveals that $\grad^{2} U^{\epsilon}_{i} \leq  \frac{1}{\epsilon} I_{n \times n}$ (see, for example \cite{KlartagPutterman23}).

As  $\epsilon \downarrow 0$, we have $ D ( X^{\epsilon}_{i} \Vert Z)  \rightarrow D ( X_{i} \Vert Z)$ for $i=0,1$. This can be seen from the expression 
\[
D ( X^{\epsilon}_{i} \Vert Z)=\int \rho^{\epsilon}_{i} \log \rho^{\epsilon}_{i} \d x + \frac{1}{2} \E \vert X^{\epsilon}_{i} \vert_2^2 + \frac{n}{2} \log (2 \pi),  \]
and \cite[Remark 10]{MadimanWang14}.

Now choose a decreasing sequence of $\epsilon_k$ that converges to 0, note that as $k \rightarrow \infty$, $\mu^{\epsilon_k}_{i}$ converges weakly to $\mu_{i}$ for $i=0,1$, respectively, therefore by Prokhorov's theorem, both sequences $\{\mu^{\epsilon_k}_{0}\}$ and $\{\mu^{\epsilon_k}_{1} \}$ are tight in $\mathcal{P}(\R^n)$. Suppose that for each $k$,  $\pi^{\epsilon_k}$ is the coupling (that is, the joint distribution of $(X_{0}^{\epsilon_k}, X_{1}^{\epsilon_k})$) that we have used in the previous part of the proof, one can show that $\{ \pi^{\epsilon_k} \}$ is also a tight sequence in $\mathcal{P}(\R^n \times \R^n)$, whence it admits a weakly convergent subsequence, and without loss of generality, we assume that $(\pi^{\epsilon_k})$ converges to a coupling $\pi$ of $\mu_{0}$ and $\mu_{1}$. Since the relative entropy $D(\mu \Vert \nu)$ is lower semicontinuous on $\mathcal{P}(\R^n) \times \mathcal{P}(\R^n)$, where $\mathcal{P}(\R^n)$ is equipped with the weak topology, we have that 
\[
\liminf_{k \rightarrow \infty} D([(x,y) \mapsto (1-t)x+ty]_{\#}\pi^{\epsilon_k} \Vert \gamma) \geq D([(x,y ) \mapsto (1-t)x+ty]_{\#} \pi \Vert \gamma).
\]
With the last observation, and that we already have
\begin{equation} 
    e^{-\frac{1}{n} D (X^{\epsilon_k}_{t} \Vert Z)} \geq \sigma^{(1-t)} \left( \theta_{\epsilon_k}\right) e^{-\frac{1}{n} D (X^{\epsilon_k}_{0} \Vert Z)} + \sigma^{(t)} \left(\theta_{\epsilon_k} \right)  e^{-\frac{1}{n} D (X_{1} \Vert Z)},
\end{equation}
for $X^{\epsilon_k}_{t}=(1-t) X^{\epsilon_k}_{0} + t X^{\epsilon_k}_{1}$ and $\theta_{\epsilon_k} = \left( \E \vert X^{\epsilon_k}_{0} - X^{\epsilon_k}_{1} \vert^{2} \right)^{\frac{1}{2}}$, we can send
$k \to \infty$ to complete the proof of both Theorem \ref{thm: mainEnt} and the claim in Remark \ref{rem: curvaturemain}.
\end{proof}
\begin{rem} \label{rem: notagradientfield}
     The vector fields that appear in \cite{AishwaryaRotem23} are gradient fields, which simplifies things. For example, if $u_{t}$ in the proof above was a gradient field, one could simply apply the Poincar\'e inequality (with respect to $\mu_{t}$) to the components of $u_{t}$ to get $\int \tr (\grad u_{t})^{2} \d \mu_{t} \geq \int \vert u_{t} \vert^{2} \d \mu_{t}$.
\end{rem}

\begin{proof}[Proof of Theorem \ref{thm: mainbbl}]
    Set $F = \log f, G = \log g,$ and $H = \log h$. We will use H\"older's inequality in the form
    \begin{equation}
        M_{p}^{t} (a,b) M_{q}^{t} (x,y) \geq M_{r}^{t} (ax, by),
    \end{equation}
    where $q = \frac{1}{n}$ and $\frac{1}{p} + \frac{1}{q} = \frac{1}{r}$,
applied to
\begin{equation}
    x= e^{- D(\mu_{0} \Vert \gamma)}, y = e^{- D(\mu_{1} \Vert \gamma)} , a = e^{\int F \d \mu_{0}}, b = e^{\int G \d \mu_{1} },
  \end{equation}  
  where $\d \mu_{0} \propto f \d \gamma$, and $\d \mu_{1} \propto g \d \gamma$ are probability measures.
  Thus, we get
  \begin{equation} \label{eq: holderapplied}
      \begin{split}
          & \left( (1-t) e^{p \int F \d \mu_{0}} + t e^{p \int G \d \mu_{1}}    \right)^{\frac{1}{p}} \left( (1-t) e^{ - \frac{1}{n} D (\mu_{0} \Vert \gamma)} + t e^{ - \frac{1}{n} D (\mu_{1} \Vert \gamma)}    \right)^{\frac{1}{q}}  \\
    &\geq \left(   (1-t) e^{r \left(  \int F \d \mu_{0} - D(\mu_{0} \Vert \gamma ) \right)} + t e^{r \left(  \int G \d \mu_{1} - D(\mu_{1} \Vert \gamma ) \right)} \right)^{\frac{1}{r}} \\
    &=  \left( (1-t) e^{r \log \int e^{F} \d \gamma} + t e^{r \log \int e^{G} \d \gamma} \right)^{\frac{1}{r}} \\
    &= \left( (1-t) \left( \int e^{F} \d \gamma \right)^{r} + t \left(  \int e^{G} \d \gamma \right)^{r} \right)^{\frac{1}{r}}.
      \end{split}
  \end{equation}
  Taking the expectation of 
        \begin{equation}
      H ((1-t) X_{0} +  t X_{1}) \geq \log \left(  (1-t) e^{pF(X_{0})} + t e^{p G(X_{1})} \right)^{\frac{1}{p}},
  \end{equation}
 with respect to any joint distribution of $(X_{0}, X_{1})$ and combining it with the joint convexity of the function 
  \begin{equation}
      \Psi_{t} (u,v) = \log \left(  (1-t) e^{u} + t e^{v}  \right)
  \end{equation}
  for every fixed $t$ \cite[Lemma 2.11]{ErbarKuwadaSturm15}, we see that
  \begin{equation} \label{eq: funcpieceDV}
      e^{\int H \d \mu_{t}} \geq \left( (1-t) e^{p \int F \d \mu_{0}} + t e^{p \int G \d \mu_{1}}    \right)^{\frac{1}{p}}.
  \end{equation}
 Moreover, we already have
 \begin{equation} \label{eq: entpieceDV}
     e^{ - \frac{1}{n} D(\mu_{t} \Vert \gamma)} \geq (1-t) e^{ - \frac{1}{n} D (\mu_{0} \Vert \gamma)} + t e^{ - \frac{1}{n} D (\mu_{1} \Vert \gamma)},
\end{equation}
  when $\mu_{t}$ is the distribution of $X_{t} = (1-t)X_{0} + t X_{1}$ for the joint distribution $(X_{0}, X_{1})$ from Theorem \ref{thm: mainEnt}. Putting equations \eqref{eq: funcpieceDV} and \eqref{eq: entpieceDV} together and invoking the ``inequality part'' of Donsker--Varadhan duality, we get
  \begin{equation}
      \begin{split}
      & \int e^{H} \d \gamma  \geq e^{\int H \d \mu_{t}}   e^{ -  D(\mu_{t} \Vert \gamma)}\\
      & \geq \left( (1-t) e^{p \int F \d \mu_{0}} + t e^{p \int G \d \mu_{1}}    \right)^{\frac{1}{p}} \left( (1-t) e^{ - \frac{1}{n} D (\mu_{0} \Vert \gamma)} + t e^{ - \frac{1}{n} D (\mu_{1} \Vert \gamma)}    \right)^{\frac{1}{q}}, 
      \end{split}
  \end{equation}
  which completes the proof by \eqref{eq: holderapplied}.
\end{proof}

\bibliographystyle{amsplain}
\bibliography{shoulderofgiants}
\end{document}